\documentclass{article}

\usepackage[english]{babel}

\usepackage[a4paper,top=2cm,bottom=2cm,left=3cm,right=3cm,marginparwidth=1.75cm]{geometry}

\usepackage{amsmath, amssymb}
\usepackage{graphicx}
\usepackage[colorlinks=true, allcolors=blue]{hyperref}
\usepackage{mathtools}
\usepackage{amsfonts}
\usepackage{amsthm}
\usepackage{tikz-cd}
\usepackage{polski}
\usepackage{comment}

\newtheorem{thm}{Theorem}[section]

\newtheorem{cor}[thm]{Corollary}
\newtheorem{lem}[thm]{Lemma}
\newtheorem{prop}[thm]{Proposition}
\newtheorem{rem}[thm]{Remark}
\newtheorem{construction}[thm]{Construction}

\numberwithin{equation}{section}
\newcommand{\mycomment}[1]{}

\newcommand{\tC}{\Tilde{C}}

\newcommand{\s}{\sigma}

\usepackage[textwidth=0.8in]{todonotes}
\DeclareMathOperator{\Aut}{Aut}

\DeclareMathOperator{\Fix}{Fix}

\DeclareMathOperator{\Nm}{Nm}
\DeclareMathOperator{\Pic}{Pic}
\DeclareMathOperator{\End}{End}
\DeclareMathOperator{\Ker}{Ker}

\DeclareMathOperator{\Image}{Im}
\DeclareMathOperator{\Spec}{Spec}

\newcommand{\unnumberedfootnote}[1]{%
  \begingroup
  \renewcommand{\thefootnote}{}%
  \footnotetext[0]{#1}%
  \endgroup
}

\title{On the Prym map of degree 4 cyclic covers of hyperelliptic curves}
\author{Anatoli Shatsila}
\begin{document}
\maketitle
\begin{abstract}

In this paper, we study the Prym map associated to degree 4 \'etale cyclic covers of genus $g$ hyperelliptic curves restricted to the irreducible component $\mathcal{RH}_g[4]^{hyp}$ of the moduli space of such covers where an intermediate cover is hyperelliptic. We show that for $g \geq 3$ the Prym map is injective on $\mathcal{RH}_g[4]^{hyp}$. In the case $g=2$ (where $\mathcal{RH}_2[4]^{hyp} = \mathcal{RH}_2[4]$) we prove that non-empty fibers of the Prym map, apart from two exceptional fibers, are isomorphic to the projective line without 8 points. Moreover, we obtain a new description of the space $\mathcal{RH}_g[4]^{hyp}$ in terms of tuples of complex numbers and find equations of hyperelliptic curves arising from such covers. 

\end{abstract}

\unnumberedfootnote{\textit{2020 Mathematics Subject Classification:} 14H30, 14H40, 14H45, 14H55, 14K12}
\unnumberedfootnote{\textit{Key words and phrases:} Prym variety, Prym map, coverings of curves.}

\section{Introduction}

Given a finite cover of smooth curves $f: C \to H$, the associated Prym variety is defined as the principal component of the kernel of the norm map $$P(f) := \Ker^0[\Nm_f: \Pic^0(C) \to \Pic^0(H)].$$ 

The classical case when $f$ is a double \'etale cover of a genus 2 curve $H$ has been considered in \cite{MR379510}. In recent years, cyclic étale covers of hyperelliptic curves have attracted considerable attention from researchers \cite{MR2023955, MR3519095,  MR4093071, MR4747961, naranjo2024simplicityjacobiansautomorphisms, borówka2025prymmapscycliccoverings}. 

Let $f: C \to H$ be a degree $d \geq 2$ \'etale cyclic cover of a hyperelliptic genus $g$ curve $H$. Then the Prym variety $P(f)$ is a $(d-1)(g-1)$-dimensional abelian variety with polarization of type $(1,1,\ldots, 1, d, \ldots, d)$, where $d$ appears $g-1$ times. In the modular setting, we get a Prym map $$\mathcal{P}_g[d]: \mathcal{RH}_g[d] \to \mathcal{A}_{(d-1)(g-1)}^{(1,\ldots,1,d,\ldots,d)}$$ which associates the Prym variety $P(f)$ to a cover $[f: C \to H] \in \mathcal{RH}_g[d]$, where $$\mathcal{RH}_g[d] := \{(H, \langle\eta\rangle) \: | \: H \text{ is hyperelliptic of genus } g; \eta\in JH[d], \langle \eta \rangle \simeq \mathbb{Z}_d\}/\simeq.$$ 

It is known that the map $\mathcal{P}_g[d]$ is injective for any $g \geq 2$ and $d \geq 6$ if $d$ is not a prime power and generically injective if $d$ is a power of a prime but not a prime \cite{borówka2025prymmapscycliccoverings} or $d$ is a prime such that $(d-1)(g-1) \geq 7$ and $g \not\equiv 3 \pmod d$ \cite{naranjo2024simplicityjacobiansautomorphisms}. Moreover, $\mathcal{P}_g[d]$ is generically finite for $g \geq 3$, $d \geq 5$ and $g=2, d\geq 6$ \cite{naranjo2024simplicityjacobiansautomorphisms, MR4093071}. 

The case $g=2$ has been studied the most extensively. The map $\mathcal{P}_2[d]$ is dominant of degree 10 for $d=7$ \cite{MR3519095} and has positive-dimensional fibers for $d=3,4,5$ \cite{MR3522345, borowka2025prymsmathbbz3timesmathbbz3coveringsgenus}. Prym varieties of degree 3 cyclic covers of genus 2 curves have been studied in many different contexts \cite{BLiso, BLEE, ries}; however, the author is not aware of a result in the literature that provides a complete characterization of the fibers of $\mathcal{P}_2[3]$. An explicit description of the fibers of $\mathcal{P}_2[5]$ appears to be unknown as well.  

In this paper, we focus on the map $$\mathcal{P}_g[4]: \mathcal{RH}_g[4]^{hyp} \to \mathcal{A}_{3g-3}^{(1,\ldots,1,4,\ldots, 4)},$$ where $$\mathcal{RH}_g[4]^{hyp} := \{(H, \langle\eta\rangle) \: | \: H \text{ is hyperelliptic of genus } g; \eta \in JH[4]\setminus JH[2] \text{ with } \eta^2 = \mathcal{O}_H(w_1 - w_2)\}/\simeq$$ is an irreducible component of $\mathcal{RH}_g[4]$ such that the curve $\Spec(\mathcal{O}_H \oplus \eta^2)$ is hyperelliptic for any $[H, \langle \eta \rangle] \in \mathcal{RH}_g[4]^{hyp}$ (see Lemma \ref{hyperellipticmiddlecurve}). Note that $\mathcal{RH}_2[4]^{hyp} = \mathcal{RH}_2[4]$ by Corollary \ref{genus3hyperelliptic}.

The main results of the manuscript can be summarised in the following theorem.

\begin{thm}[Theorem \ref{discr}, Theorem \ref{injective}]

i) The Prym map $\mathcal{P}_g[4]: \mathcal{RH}_g[4]^{hyp} \to\mathcal{A}_{3g-3}^{(1,\ldots,1,4,\ldots, 4)} $ is injective for $g \geq 3$.

ii) Each non-empty fiber of $\mathcal{P}_2[4]$, except for two exceptional fibers, is isomorphic to the projective line without 8 points.

\end{thm}

There are two main components of the proof that are developed in Sections \ref{sec2} and \ref{sec3}, respectively. First, we show that elements of $\mathcal{RH}_g[4]^{hyp}$ are in bijective correspondence with elements of the set $$\Delta_g := \{\{t_1,\ldots, t_{2g-1}\} \subset \mathbb{C}\setminus \{0,1\} \ | \ t_i^2 \neq t_j^2 \text{ for any } i\neq j\}/\sim,$$ where \(\{t_1, \ldots, t_{2g-1}\} \sim \{s_1, \ldots, s_{2g-1}\}\) if and only if $\{t_1,\ldots, t_{2g-1}\} = \{s_1, \ldots, s_{2g-1}\}$ or $\{t_1, \ldots, t_{2g-1}\} = \{\frac{1}{s_1}, \ldots, \frac{1}{s_{2g-1}}\}$. This is done by a careful analysis of order two line bundles defining the cover in terms of Weierstrass points of the hyperelliptic genus $g$ curve $H$. We use the equations of hyperelliptic curves $C_j, C_{j\sigma^2}, C_{j\s, \s^2}$ arising as quotients of $C$ by the involutions to describe the data of line bundles as elements of $\Delta_g$. In Section \ref{sec3}, we study the isogeny $$\mu: JC_j \times JC_{j\sigma^2} \times JC_{j\s,\s^2} \to P(f).$$ In particular, we show that it can be intrinsically recovered from the Prym variety $P(f)$ and its induced polarization. In Section \ref{sec4}, we provide a description of $\mathcal{P}_2[4]$ as the map $$\mathcal{P}_2[4]: \Delta_2 \to \mathcal{A}_3^{(1,1,4)}$$ and show that its fibers are isomorphic to subsets of $\Delta_2$ of the form $$\Delta_{\lambda_1,\lambda_2} := \left\{\{t_1, t_2, t_3\} \in \Delta_2 \: | \: \lambda_1 = \frac{(t_2 - 1)(t_3 - t_1)}{(t_2 - t_1)(t_3 - 1)}, \: \lambda_2 = \frac{(t_2^2 - 1)(t_3^2 - t_1^2)}{(t_2^2 - t_1^2)(t_3^2 - 1)}\right\}$$ for some $\lambda_1 \neq \lambda_2$. Finally, we prove that any such set, apart from two exceptions $\Delta_{-1, \frac{1}{2}}, \Delta_{e^{\pi i /3}, e^{-\pi i /3}}$, is isomorphic to the projective line without 8 points. In Section \ref{sec5}, we use the map $\mu$ and equations of $C_j, C_{j\s^2}, C_{j\s,\s^2}$ to show that the Prym map $\mathcal{P}_g[4]$ is injective on $\mathcal{RH}_g[4]^{hyp}$.

Our methods do not directly generalise to other irreducible components of $\mathcal{RH}_g[4]$. For a cover $[f:C \to H] \in  \mathcal{RH}_g[4] \setminus \mathcal{RH}_g[4]^{hyp}$, the curve $C_{j,\s^2}$ is not isomorphic to $\mathbb{P}^1$, so the curves $C_{j}$ and $C_{j\s^2}$ are not hyperelliptic and there is no decomposition of $P(f)$ as the product of three Jacobians as in Lemma \ref{isotypical}.

We note that while similar approaches --- such as describing moduli spaces of covers using tuples of points on the projective line \cite{MR4068284} or decomposing Prym varieties into elliptic curves and employing their equations \cite{MR4794691} --- have been successfully applied to other Prym maps, this is the first paper to utilize these methods for describing positive-dimensional fibers. 

\subsection*{Acknowledgements}
The author has been supported by the Polish National Science Center project number 2024/53/N/ST1/01634. The author thanks Paweł Borówka for valuable suggestions that improved the article. 

\section{Degree 4 cyclic covers of hyperelliptic curves}
\label{sec2}

The aim of this section is to study \'etale degree 4 cyclic covers of a hyperelliptic curve $H$ of genus $g$ from the perspective of Weierstrass points of $H$. We start by recalling the main facts from the theory of double covers of hyperelliptic curves that can be found in \cite[Chapter 5.2.2.]{MR2964027}.

\begin{lem}

    Let $H$ be a hyperelliptic curve of genus $g$ and $W = \{w_1,\ldots,w_{2g+2}\}$ be the set of its Weierstrass points. 

    \begin{itemize}
        \item[i)] For any point $\eta \in JH[2]$ there exist two disjoint subsets $A_+, A_- \subset \{1,\ldots,2g+2\}$ of equal cardinality such that $$\eta = \mathcal{O}_H\left(\sum_{i \in A_+}w_i - \sum_{j \in A_-}w_j\right).$$
        \item[ii)] For any $A \subset \{1,\ldots,2g+2\}$ with $|A| = g+1$ the following equality holds in $\Pic(H)^{g+1}:$ $$\mathcal{O}_H\left(\sum_{i \in A} w_i\right) = \mathcal{O}_H\left(\sum_{j \in \{1,2,\ldots,2g+2\}\setminus A}w_j\right).$$ 
    \end{itemize}
    In particular, for $g = 2$, any two-torsion point of $JH$ is given by $\mathcal{O}_H(w_i - w_j)$ for some $i,j$. 
\end{lem}

The following lemma provides a criterion for the covering curve to be hyperelliptic. 

\begin{lem}[\cite{BO17}] 
\label{hyperellipticmiddlecurve}
    Let \(H\) be a hyperelliptic curve of genus \(g\) and \(h: C \rightarrow H\) an étale double covering defined by \(\eta \in J H[2]\). Then \(C\) is hyperelliptic if and only if \(\eta=\mathcal{O}_H\left(w_1-w_2\right)\), where \(w_1, w_2 \in H\) are Weierstrass points.
\end{lem}

In particular, we have the following result, already known to Mumford \cite{MR379510}.

\begin{cor}
\label{genus3hyperelliptic}
    Let $g: C \to H$ be an \'etale double cover of a genus 2 curve $H$. Then $C$ is hyperelliptic.
\end{cor}

Recall that the space of degree 4 \'etale covers of genus \(g\) hyperelliptic curves is defined as $$\mathcal{RH}_g[4] := \{(H, \langle\eta\rangle) \: | \: H \text{ is hyperelliptic of genus } g; \eta \in JH[4]\setminus JH[2]\}/\simeq.$$ This space has a natural stratification given by the presentation of the two-torsion point $\eta^2$ in terms of the Weierstrass points of $H$. For $1 \leq m \leq \lfloor \frac{g+1}{2}\rfloor$ we define $$\mathcal{RH}_g[4]_{m} := \{(H, \langle\eta\rangle) \: | \: H \text{ is hyperelliptic of genus } g; \eta \in JH[4]\setminus JH[2] \text{ with } \eta^2 = \mathcal{O}_H(\sum_{i=1}^{2m}(-1)^iw_i)\}/\simeq.$$ Then we have $$\mathcal{RH}_g[4] = \bigsqcup_{i=1}^{\lfloor \frac{g+1}{2}\rfloor} \mathcal{RH}_g[4]_m.$$ 
We also introduce the notation $\mathcal{RH}_g[4]^{hyp} := \mathcal{RH}_g[4]_1$. It follows from Lemma \ref{genus3hyperelliptic} that $\mathcal{RH}_2[4]^{hyp} = \mathcal{RH}_2[4]$.  

Let $H$ be a hyperelliptic genus $g$ curve and $W$ be the set of Weierstrass points of $H$. Let $f: C \to H$ be an \'etale cyclic cover of degree 4 given by a line bundle $\eta \in JH[4]\setminus JH[2]$ with $\eta^2 = \mathcal{O}_H\left(\sum_{i=1}^{2m}(-1)^iw_i\right)$ for $1 \leq m \leq \lfloor \frac{g+1}{2}\rfloor$ and $w_i \in W$, and $\sigma$ be an automorphism of $C$ inducing $f$. Then $g(C) = 4g-3$ and by \cite{ries} the hyperelliptic involution $\iota$ on $H$ lifts to four involutions on $C$ denoted by $j, j\s, j\s^2, j\s^3$. The automorphisms $\sigma$ and $j$ generate the dihedral group $D_4 \subset \Aut(C)$. Since there are two conjugacy classes of involutions in $D_4$ given by $\{j, j\s^2\}$ and $\{j\s, j\s^3\}$ it follows that the quotient curve $C_{j} := C/\langle j\rangle$ is isomorphic to $C_{j\s^2} := C/\langle j\s^2\rangle$ and $C_{j\s} := C/\langle j\s\rangle$ is isomorphic to $C_{j\s^3} := C/\langle j\s^3\rangle$. 

Let $W_{\eta} = \{w_1,\ldots, w_{2m}\}$ be the subset of Weierstrass points of $H$ defining $\eta$ and $W^c_{\eta} = W \setminus W_{\eta}$ be its complement. The curve $C_{\sigma^2} := C/\langle \s^2 \rangle$ is a curve of genus $2g-1$ which is hyperelliptic if $[f:C \to H] \in \mathcal{RH}_g[4]^{hyp}$. Let $k: C_{\s^2} \to H$ be the covering map and $j, j\sigma$ be the lifts of $\iota$ to $C_{\s^2}$. Then, up to exchanging $j$ and $j\sigma$, we have $\Fix(j\s) = g^{-1}(W_{\eta})$ and $\Fix(j) = g^{-1}(W_{\eta}^c)$. Thus, there is the following diagram, where all maps except $f$ are double covers: 

\begin{center}
\label{diagram1}
\begin{tikzcd}
                                            &  & C \arrow[d, "h"] \arrow[rrd] \arrow[dd, "f"', bend right] \arrow[lld] &  &                                       \\
C_{j\sigma}  \arrow[d]           &  & C_{\sigma^2} \arrow[lld] \arrow[rrd] \arrow[d, "k"]                   &  & C_{j} \arrow[d]             \\
{C_{\sigma^2, j\sigma}} \arrow[rrd] &  & H \arrow[d]                                                           &  & {C_{\sigma^2, j} } \arrow[lld] \\
                                            &  & \mathbb{P}^1                                                          &  &                                      
\end{tikzcd}   
\end{center}

Using the fact that $|\Fix(j)| = |\Fix(j\sigma^2)|$, $|\Fix(j\s)| = |\Fix(j\sigma^3)|$ and the Riemann-Hurwitz formula, we get $g(C_{j\s}) = 2g-m-1, g(C_{j}) = g+m-2, g(C_{\sigma^2, j}) = m-1, g(C_{\sigma^2, j\s}) = g-m$. 

In the rest of the section we restrict our attention to covers $[f: C \to H] \in \mathcal{RH}_2[4]^{hyp}$. Note that in this case $C_{\sigma^2,j}$ is isomorphic to $\mathbb{P}^1$ so $C_j$ is hyperelliptic. We see that $f$ is constructed from line bundles $\nu \in JH[2]$ and $\xi \in JC_{\sigma^2}[2]$ such that there is $\eta \in JH[4] \setminus JH[2]$ with $\eta^2 = \nu$ and $k^*\eta = \xi$, where $k: C_{\sigma^2} \to H$ is a double cover given by $\nu$. 
In the construction below, we will describe such covers using the Weierstrass points on $H$. With a slight abuse of notation, we often identify Weierstrass points on a curve with their images in $\mathbb{P}^1$ under the double cover given by the hyperelliptic involution.

\begin{prop}
\label{bundles}
    Let $H$ be a genus $g$ hyperelliptic curve, $\{u_1, u_2, w_1, \ldots w_{2g}\}$ be the set of its Weierstrass points, $\nu = \mathcal{O}_H(u_1-u_2)$, and $k: C_{\sigma^2} \to H$ be the cover defined by $\nu$. Let $\psi: \mathbb{P}^1 \to \mathbb{P}^1$ be a double cover branched over $u_1, u_2$ and $\bar{\sigma}$ be the involution of $\mathbb{P}^1$ inducing $\psi$. With a slight abuse of notation, we denote the preimages of $w_i$ under $\psi$ by $Q_i = \{w_i, \bar\sigma(w_i)\}$. 
    Then $ \bigcup_{i=1}^{2g} Q_i$ is the set of Weierstrass points of $C_{\sigma^2}$ and, if we define $$Q := \left\{\mathcal{O}_{C_{\sigma^2}}\left(\sum_{i=1}^{2g}(-1)^iq_i\right) \ \mid \ q_i \in Q_i\right\}$$ then $$Q = \{k^*\eta \ | \ \eta^2 = \nu\}.$$
\end{prop}

\begin{proof}
    The fact that $ \bigcup_{i=1}^{2g} Q_i$ is the set of Weierstrass points of $C_{\sigma^2}$ is well-known \cite[Corollary 4.3]{BO17}. 
    
    Since $\Nm_{k}\circ k^* = m_2$, where $m_2$ is the multiplication by $2$ on $JH$, for any $\eta$ with $\eta^2 = \nu$ we have $(\Nm_{k}\circ k^*)(\eta) = \nu$. Note that $$\Nm_k\left(\mathcal{O}_{C_{\sigma^2}}\left(\sum_{i=1}^{2g}(-1)^iq_i\right)\right) = \mathcal{O}_H\left(\sum_{i=1}^{2g}(-1)^iw_i\right) = \mathcal{O}_H(u_1 - u_2) = \nu,$$ and the elements of $Q$ are the only 2-torsion points of $JC_{\sigma^2}$ with this property, which yields $\{k^*\eta \: | \: \eta^2 = \nu\} \subseteq Q$.
Moreover, $\mathcal{O}_{C_{\sigma^2}}(\sum_{i=1}^{2g}(-1)^iq_i) = \mathcal{O}_{C_{\sigma^2}}\left(\sum_{i=1}^{2g}(-1)^iq_i'\right)$ with $q_i, q_i' \in Q_i$ for all $i$ if and only if for all $i$ either $q_i = q_i'$ or $q_i \neq q_i'$, hence $|Q| = 2^{2g-1}$. Finally, there are precisely $2^{2g}$ 4-torsion points $\eta \in JH$ satisfying $\eta^2 = \nu$, and the equality $k^*\eta_1 = k^*\eta_2$ holds if and only if $\eta_1 = \eta_2 + \nu$, hence $|\{k^*\eta \ | \ \eta^2 = \nu\}| = 2^{2g-1}$, which implies $Q = \{k^*\eta \: | \: \eta^2 = \nu\}$.
\end{proof}

\begin{construction}
\label{construction}
    Using Proposition \ref{bundles}, we can view the construction of the cover $f:C \to H$ starting from $2g+2$ Weierstrass points $W = \{u_1, u_2, w_1,\ldots, w_{2g}\}$ with two distinguished points $u_1, u_2$ and a line bundle $\xi \in Q$. We take $\nu := \mathcal{O}_H(u_1-u_2)$ and $\xi := \mathcal{O}_{C_{\sigma^2}}(\sum_{i=1}^{2g}(-1)^iq_i)$ for some $q_i \in Q_i$. For $g=2$ we have the following diagram, where ramifications are indicated above each arrow. 
    \begin{center}
        \includegraphics[scale = 0.3]{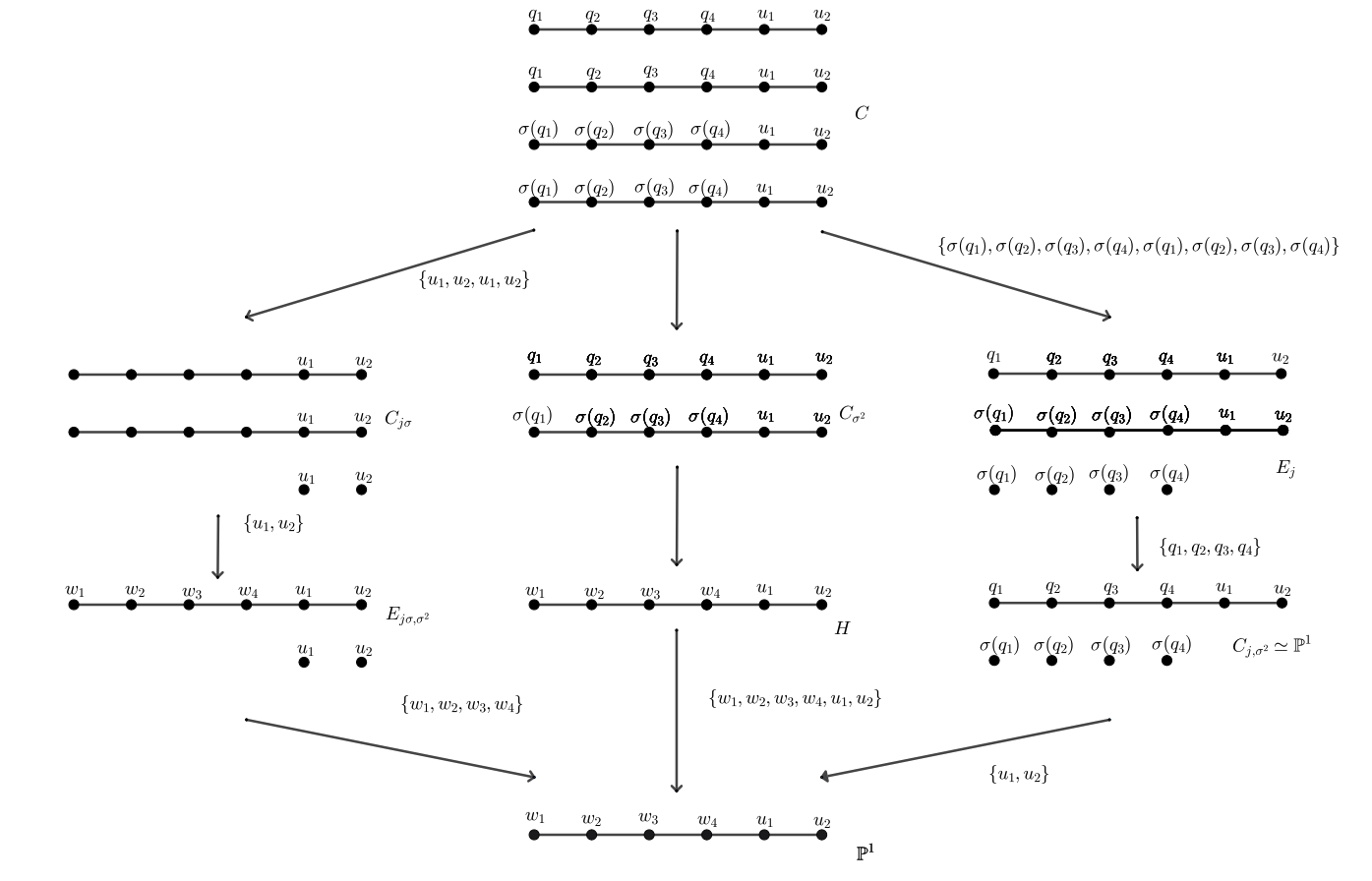}\label{diagram2}
    \end{center}
\end{construction}

Let us determine the equations of hyperelliptic curves in the diagram \ref{diagram2}. Define $$\Delta_g := \{\{t_1,\ldots, t_{2g-1}\} \subset \mathbb{C}\setminus \{0,1\} \ | \ t_i^2 \neq t_j^2 \text{ for any } i\neq j\}/\sim,$$ where \(\{t_1, \ldots, t_{2g-1}\} \sim \{s_1, \ldots, s_{2g-1}\}\) if and only if $\{t_1,\ldots, t_{2g-1}\} = \{s_1, \ldots, s_{2g-1}\}$ or $\{t_1, \ldots, t_{2g-1}\} = \{\frac{1}{s_1}, \ldots, \frac{1}{s_{2g-1}}\}$. Note that elements of $\Delta_g$ are unordered.

\begin{lem}
\label{equations}
    The hyperelliptic curves appearing in Diagram \ref{diagram2} have the following defining equations:
    \begin{align*}
    C_{\sigma^2}: \:\:\: y^2 &= (x^2 - 1)\prod_{i=1}^{2g-1}(x^2 - t_i^2), \\
    H: \:\:\: y^2 &= x(x - 1)\prod_{i=1}^{2g-1}(x - t_i^2), \\
    C_{\sigma^2,j\sigma}: \:\:\: y^2 &=(x - 1)\prod_{i=1}^{2g-1}(x - t_i^2), \\
    C_{j}: \:\:\: y^2 &=(x - 1)\prod_{i=1}^{2g-1}(x - t_i), \\
    C_{j\sigma^2}: \:\:\: y^2 &=(x + 1)\prod_{i=1}^{2g-1}(x + t_i).
    \end{align*} for some $(t_1,\ldots,t_{2g-1}) \in \Delta$.
\end{lem}

\begin{proof}
   
   By the projective change of coordinates we can assume that $w_i = [t_i^2:1]$ for $i=1,\ldots,2g-1$, $w_{2g} = [1:1]$, $u_1 = [1:0]$ and $u_2 = [0:1]$. Then the involution $\overline{\sigma}$ from Proposition \ref{bundles} is given by $[x:y] \mapsto [-x:y]$. It follows that the cover $\psi: \mathbb{P}^1 \simeq C_{\s^2, j} \to C_{\s, j} \simeq \mathbb{P}^1$ is given by $[x:1] \mapsto [x^2:1]$. Assuming $q_i = t_i$ for $i=1,\ldots,2g-1$, the equations follow from the straightforward generalization of diagram \ref{diagram2}. 
   
\end{proof}

\begin{rem}
\label{useful}
    Fixing the coordinates of the $2g+2$ points on the projective line $w_i = [s_i^2:1]$ for $i=1,\ldots,2g-1$, $w_{2g} = [1:1]$, $u_1 = [1:0]$ and $u_2 = [0:1]$, there are $2^{2g-1}$ possible curves $C_j$ given by the defining equation $$y^2 = (x-1)\prod_{i=1}^{2g-1}(x\pm s_i).$$ The defining equation of $C_{j\sigma^2}$ is then $$y^2 = (x+1)\prod_{i=1}^{2g-1}(x\mp s_i),$$ and from this it is evident that $C_j$ is isomorphic to $C_{j\sigma^2}$. 

    The choice of a line bundle in $Q$ in Proposition \ref{bundles} is equivalent to the choice of signs above.
\end{rem}

Construction \ref{construction} gives us a more convenient description of elements of the of the moduli space $\mathcal{RH}_g[4]^{hyp}$.

\begin{thm}
\label{moduli}
    There is a bijection between $\mathcal{RH}_g[4]^{hyp}$ and $\Delta_g$.
\end{thm}
\begin{proof}

It follows from Construction \ref{construction} that the data $[(H, \langle \eta \rangle)] \in \mathcal{R}_g[4]^{hyp}$ is equivalent to the data of $2g+2$ points $\{w_1,\ldots,w_{2g};u_1,u_2\}$ with a distinguished pair $\{u_1,u_2\}$ (up to projective equivalence respecting the pair) and the choice of an element in each fiber over $w_i$ of the covering $\mathbb{P}^1 \to \mathbb{P}^1$ branched at $\{u_1, u_2\}$. Under the projective change of coordinates from the proof of Lemma \ref{equations} we see that the data above is equivalent to an element $\{t_1, \ldots, t_{2g-1}\} \in \Delta_g$. Note that the equivalence relation defining $\Delta_g$ corresponds to the projective transformation that fixes \(w_{2g}\) and swaps \(u_1\) and \(u_2\).

To better illustrate the equivalence, let us describe the inverse mapping: starting from an element $\{t_1, \ldots, t_{2g-1}\} \in \Delta_g$, we define $$ w_i = [t_i:1] \text{ for } i=1,\ldots, 2g-1,  w_{2g} = [1:1], u_1 = [1:0], u_2 = [0:1]$$ and take $\{[t_1:1], \ldots, [t_{2g-1}:1], [1:1]\}$ which corresponds to the choice of an element in $Q$. 


\end{proof}

\section{Geometry of the Prym variety}
\label{sec3}

In this section, we study Prym varieties of elements of $\mathcal{RH}_g[4]^{hyp}$. We use the notation from the previous section unless explicitly stated, and we denote morphisms of Jacobians induced by automorphism of curves by the same letters. Sometimes we omit pullbacks when we consider abelian subvarieties. 
Let $$(P, \Xi) := P(f)$$ be the Prym variety of a cover $f: C \to H$ with the induced polarization $\Xi$. 

Before we determine the isotypical decomposition of $P$, let us introduce the following notation. Let $A$ be an abelian variety and $A_1, A_2 \ldots, A_n \subseteq A$ be a set of its abelian subvarieties with the induced polarizations. We denote the symmetric idempotent of $A_k$ by $\varepsilon_{k} \in \End_{\mathbb{Q}}(A)$ for $1 \leq k \leq n$. We write $$A = A_1^{(e_1)} \boxplus A_2^{(e_2)} \boxplus \ldots \boxplus A_n^{(e_n)}$$ if $e_k$ is the polarization type of \(A_k\) for $1 \leq k \leq n$ and $$\sum_{i=1}^n \varepsilon_i = 1$$. 

\begin{lem}
\label{isotypical}
Let $P = P(f)$ be the Prym variety of the cover $[f:C \to H] \in \mathcal{RH}_g[4]^{hyp}$. We have the following decomposition of $P$:
    $$P = JC_j^{(2)} \boxplus JC_{j\sigma^2}^{(2)} \boxplus JC_{j\sigma, \sigma^2}^{(4)}.$$ 
\end{lem}

\begin{proof}
    Note that $JC_j, JC_{j\sigma^2}$ and $JC_{j\sigma, \sigma^2}$ have exponents $2,2,4$ as subvarieties of $JC$, respectively, since the corresponding covers are of degree $2,2,4$. Hence, their symmetric idemponents are given by
$$\varepsilon_{C_j} = \frac{1 + j}{2},$$
$$\varepsilon_{C_{j\sigma^2}} = \frac{1 + j\sigma^2}{2},$$
$$\varepsilon_{C_{j\sigma,\sigma^2}} = \frac{1 + j\sigma + \sigma^2 + j\sigma^3}{4}.$$

Since $g(C_{\sigma^2, j}) = g(C_{\sigma, j}) =  0$ we have $$1 + j + \sigma^2 + j\sigma^2 = \sigma + \sigma^3 + j\sigma + j\sigma^3 = 0.$$

Using these facts we obtain $$\varepsilon_{C_j} + \varepsilon_{C_{j\sigma^2}} + \varepsilon_{C_{j\sigma,\sigma^2}} = \frac{3 - \sigma - \sigma^2 - \sigma^3}{4} = 1 - \varepsilon_{JC_{\sigma}} = \varepsilon_P.$$
\end{proof}

By $\mu$ we will denote the addition map $$\mu: JC_{j}\times JC_{j\s^2} \times JC_{j\s, \s^2} \to P.$$

\begin{lem}
\begin{enumerate}
\label{geomprym}
    \item[i)] With the notation from the diagram \ref{diagram2} we have
    $$\begin{aligned}
\ker{\mu} = \Bigg\{ 
& \left( \sum_{i \in A_+} w_i - \sum_{j \in A_-} w_j, \ \sum_{i \in A_+} \s(q_i) - \sum_{j \in A_-} \s(q_j), \ \sum_{i \in A_+} q_i - \sum_{j \in A_-} q_j \right) \\
& \in JC_{j\sigma, \sigma^2}[2] \times JC_{j\sigma^2}[2] \times JC_{j}[2] \\
& \;\Bigg|\; A_+, A_- \subset \{1,\ldots,2g\}, \ |A_+| = |A_-|, A_+ \cap A_- = \varnothing \Bigg\}.
\end{aligned}$$
    \item[ii)] Let $\pi$ be the quotient map $P \to P/K(\Xi)[2]$ and $\alpha = \pi \circ \mu$ be the composition. Then $$(P/K(\Xi)[2])^{(1,\ldots,1)} = \alpha(JC_{j}\times JC_{j\s^2})^{(1,\ldots,1)} \times (JC_{j\s, \s^2}/JC_{j\s, \s^2}[2])^{(1,\ldots,1)}.$$
\end{enumerate}
    
\end{lem}

\begin{proof}
    Recall that $h$ denotes the cover $C \to C_{\sigma^2}$. Note that $P(h) = JC_{j}^{(2,\ldots,2)} \boxplus JC_{j\sigma^2}^{(2,\ldots,2)}$ and $P(h)$ has polarization type $(2,\ldots,2)$. Therefore, it follows that $P(h) = JC_{j} \times JC_{j\sigma^2}$ with the induced polarizations and, in particular, $JC_{j} \cap JC_{j\sigma^2} = \{0\}$ in $P$.

    Since $JH$ has polarization of type $(1,4,\ldots, 4)$, it follows that $P$ has polarization of type $(1,\ldots, 1,4, \ldots, 4)$ where 4 appears $g-1$ times. Therefore, we have $\deg(\mu) = \frac{2^{g-1}\cdot 2^{g-1}\cdot 4^{g-1}}{4^{g-1}} = 4^{g-1}$. Note that $\mathbb{Z}_4^{2g-2} \simeq K(\Xi) \subset \mu(K(\mu^*\Xi))$ since $\mu$ is surjective. Since $\ker \mu \subset K(\mu^*\Xi) \simeq \mathbb{Z}_2^{2g-2} \times \mathbb{Z}_2^{2g-2} \times \mathbb{Z}_4^{2g-2}$
it follows that $\ker \mu \simeq \mathbb{Z}_2^{2g-2}$.

Denoting $\pi_j:C \to C_j, \pi_{j\sigma^2}: C \to C_{j\sigma^2}, \pi_{j\sigma, \sigma^2}: C \to C_{j\sigma, \sigma^2}$
    one readily checks that for any two disjoint subsets $A_+, A_- \in \{1,\ldots,2g\}$ of equal cardinality we have $$\pi_j^*\left(\sum_{i \in A_+} q_i - \sum_{j \in A_-} q_j\right) + \pi^*_{j\sigma^2}\left(\sum_{i \in A_+} \s(q_i) - \sum_{j \in A_-} \s(q_j)\right) + \pi_{j\sigma, \sigma^2}^*\left(\sum_{i \in A_+} w_i - \sum_{j \in A_-} w_j\right) = 0$$ so the first part of the lemma is proved. 

In order to prove the second part, we denote the corresponding polarizations on Jacobians of the decomposition of $P$ by $\Xi_{j}, \Xi_{j\sigma^2}, \Xi_{j\sigma,\sigma^2}$ and introduce the notation $$K(\Xi_{j}) = JC_{j}[2] = \langle p_1, \ldots, p_{2g-2} \rangle,$$
$$K(\Xi_{j\sigma^2}) = JC_{j\sigma^2}[2] = \langle r_1, \ldots, r_{2g-2} \rangle,$$
$$K(\Xi_{j\sigma, \sigma^2}) = JC_{j\s, \s^2}[4] = \langle v_1, \ldots, v_{2g-2} \rangle,$$
$$K(\Xi) =  \langle s_1, \ldots, s_{2g-2} \rangle.$$

Then for any $j \in \{1,\ldots, 2g-2\}$ we have $$s_j = \sum_{i=1}^{2g-2}a_i^jp_i + \sum_{i=1}^{2g-2}b_i^jr_i + \sum_{i=1}^{2g-2}c_i^jv_i$$ for some integers $a_i^j, b_i^j, c_i^j$ with $i, j \in \{1,\ldots, 2g-2\}$. In particular, $$2s_j =  2\sum_{i=1}^{2g-2}c_i^jv_i$$ Therefore, $\mathbb{Z}_2^{2g-2} \simeq K(\Xi)[2] = \langle 2s_1, \ldots, 2s_{2g-2}\rangle = \langle 2v_1, \ldots, 2v_{2g-2} \rangle = JC_{j\s, \s^2}[2]$, and $\alpha_{|JC_{j\s, \s^2}}$ is the multiplication by 2. Thus, it follows from the first part of the lemma that, up to the change of generators of $JC_{j\sigma^2}[2]$, we have  $$\ker \alpha = \langle p_1+r_1, \ldots, p_{2g-2}+r_{2g-2}, 2v_1, \ldots, 2v_{2g-2}\rangle \simeq \mathbb{Z}_2^{4g-4}.$$ This implies that $\alpha(JC_{j} \times JC_{j\s^2}) \cap \alpha(JC_{j\s,\s^2}) = \{0\}$, which finishes the proof.

\end{proof}

\section{Fibers of the Prym map in genus 2}
\label{sec4}

Let $\mathcal{A}^{(1,1,4)}_3$ be the moduli space of $(1,1,4)$-polarized abelian threefolds.  In this section, we will describe the structure of the fibers of the Prym map $$\mathcal{P}_2[4]: \mathcal{R}_{2}[4] \to \mathcal{A}_3^{(1,1,4)}.$$ 
Note that for $g=2$ the curves $C_j, C_{j\s^2}, C_{j\sigma, \sigma^2}$ defined in Section 2 are elliptic curves, so we denote them by $E_j, E_{j\s^2}, E_{j\sigma, \sigma^2}$ respectively. Moreover, we often identify elliptic curves with their Jacobians.

Recall that in Theorem \ref{moduli} we showed that there is a bijection between $\mathcal{R}_{2}[4]$ and $$\Delta_2 = \{\{t_1,t_2,t_3\} \subset \mathbb{C}\setminus \{0,1\} \ | \ t_1^2 \neq t_2^2 \neq t_3^2 \neq t_1^2\}/\sim,$$ where \(\{t_1, t_2, t_3\} \sim \{s_1, s_2, s_3\}\) if and only if $\{t_1, t_2, t_3\} = \{s_1, s_2, s_3\}$ or $\{t_1, t_2, t_3\} = \{\frac{1}{s_1}, \frac{1}{s_2}, \frac{1}{s_3}\}$. It follows from Construction \ref{construction}, Lemma \ref{equations} and Remark \ref{useful} that the Prym map $\mathcal{P}_2[4]$ is given by $$\mathcal{P}_2[4]: \Delta_2 \ni T = \{t_1,t_2,t_3\} \mapsto ((E_T \times E_T \times F_T)/\langle e_1^T + e_1^T + f_1^T, e_2^T + e_2^T + f_2^T\rangle, \Xi_T) \in \mathcal{A}_3^{(1,1,4)},$$ where the pullback of $\Xi_T$ to $E_T \times E_T \times F_T$ is isomorphic to the product polarization $ \mathcal{O}_{E_T}(2) \otimes \mathcal{O}_{E_T}(2) \otimes \mathcal{O}_{F_T}(4)$. Here, $E_T$ and $F_T$ are elliptic curves defined by the equations $$y^2 =(x - 1)(x - t_1)(x - t_2)(x - t_3)$$ and $$y^2 =(x - 1)(x - t_1^2)(x - t_2^2)(x - t_3^2)$$ respectively with $$e_1^T = [(1,0)-(t_1,0)], \quad e_2^T = [(1,0) - (t_2,0)],$$ $$f_1^T = [(1,0)-(t_1^2,0)], \quad f_2^T = [(1,0)-(t_2^2,0)].$$

First, we prove the following important result.

\begin{prop}
\label{important}
    The equality $\mathcal{P}_2[4](T) = \mathcal{P}_2[4](S)$ with $T = [\{t_1,t_2,t_3\}], S = [\{s_1, s_2, s_3\}]$ holds if and only if there is a choice of indices of elements of $S$ such that $$\frac{(t_2 - 1)(t_3 - t_1)}{(t_2 - t_1)(t_3 - 1)} = \frac{(s_2 - 1)(s_3 - s_1)}{(s_2 - s_1)(s_3 - 1)}$$ and $$
        \frac{(t_2^2 - 1)(t_3^2 - t_1^2)}{(t_2^2 - t_1^2)(t_3^2 - 1)} =  \frac{(s_2^2 - 1)(s_3^2 - s_1^2)}{(s_2^2 - s_1^2)(s_3^2 - 1)}.$$
 
\end{prop}

\begin{proof}
    The equalities $$\frac{(t_2 - 1)(t_3 - t_1)}{(t_2 - t_1)(t_3 - 1)} =  \frac{(s_2 - 1)(s_3 - s_1)}{(s_2 - s_1)(s_3 - 1)}$$ and $$
        \frac{(t_2^2 - 1)(t_3^2 - t_1^2)}{(t_2^2 - t_1^2)(t_3^2 - 1)} = \frac{(s_2^2 - 1)(s_3^2 - s_1^2)}{(s_2^2 - s_1^2)(s_3^2 - 1)}$$ imply that there exist isomorphisms $\phi_1: E_{T} \to E_S$ and $\phi_2: F_T \to F_S$ such that $$\phi(e_1^T) = e_1^S, \quad \phi(e_2^T) = e_2^S, \quad \phi(f_1^T) = f_1^S, \quad \phi(f_2^T) = f_2^S.$$ Thus, $\mathcal{P}_2[4](T) = \mathcal{P}_2[4](S)$ from the definition.

        For the opposite implication, let $(P, \Xi) = \mathcal{P}_2[4](S) = \mathcal{P}_2[4](T)$ and 
        $f: C \to H$ be the covering given by $T$. Below we use the notation from the previous sections. It follows from Lemma \ref{geomprym} that $P/\Xi[2]$ is a principally polarised abelian threefold isomorphic to a product of an elliptic curve $F = \alpha(E_{j\s, \s^2})$ and an abelian sufrace $S = \alpha(E_j \times E_{j\s^2})$. Note that by \cite[Decomposition theorem 4.3.1]{BL} the isomorphism above is unique. Using the notation from Diagram \ref{diagram1} and Lemma \ref{geomprym}, we see that $P(h)^{(2,2)} = E_j^{(2)} \times E_{j\s^2}^{(2)}$ is a complementary abelian subvariety to $\pi^{-1}(F) \simeq E_{j\s, \s^2}$ in $P$. Note that the intersection $P(h) \cap \pi^{-1}(E_{j\s, \s^2})$ gives the kernel $\ker \mu$ of the addition map $$ \mu: E_j \times E_{j\s^2} \times E_{j\s, \s^2} \to P.$$ This shows that we can recover the map $\mu$ from $(P, \Xi)$ independently of the cover in its preimage. Therefore, there exist isomorphisms $\phi_1: E_T \to E_S$ and $\phi_2: F_T \to F_S$ such that $$\phi(e_1^T) = e_1^S, \quad \phi(e_2^T) = e_2^S, \quad \phi(f_1^T) = f_1^S, \quad \phi(f_2^T) = f_2^S,$$ hence we are done.
        
\end{proof}

\begin{lem}\label{difflambdas}
    Let $$\lambda_1 := \frac{(t_2 - 1)(t_3 - t_1)}{(t_2 - t_1)(t_3 - 1)}$$ and $$\lambda_2 := 
        \frac{(t_2^2 - 1)(t_3^2 - t_1^2)}{(t_2^2 - t_1^2)(t_3^2 - 1)}$$ for $\{t_1,t_2,t_3\} \in \Delta_2$. Then $\lambda_1 \neq \lambda_2$.
\end{lem}

\begin{proof}
    The equations above are equivalent to  $$\begin{cases}\label{syst}
        \frac{(t_2 - 1)(t_3 - t_1)}{(t_2 - t_1)(t_3 - 1)} = \lambda_1, \\
        \frac{(t_2 + 1)(t_3 + t_1)}{(t_2 + t_1)(t_3 + 1)} = \frac{\lambda_2}{\lambda_1}.
    \end{cases}$$
    Note that for $\lambda_1 = \lambda_2$ the second equation transforms into $$(t_2 - t_3)(t_1 - 1) = 0,$$ so any solution of system \ref{syst} is not contained in $\Delta_2$.

   
\end{proof}

Together, Proposition \ref{important} and Lemma \ref{difflambdas} imply

\begin{cor}
    Each fiber of $\mathcal{P}_2[4]$ is isomorphic to $$\Delta_{\lambda_1,\lambda_2} := \left\{[\{t_1, t_2, t_3\}] \in \Delta_2 \: | \: \lambda_1 = \frac{(t_2 - 1)(t_3 - t_1)}{(t_2 - t_1)(t_3 - 1)}, \: \lambda_2 = \frac{(t_2^2 - 1)(t_3^2 - t_1^2)}{(t_2^2 - t_1^2)(t_3^2 - 1)}\right\}$$ for some $\lambda_1 \neq \lambda_2$. 
\end{cor}

From now on, we will identify fibers of $\mathcal{P}_2[4]$ with sets $\Delta_{\lambda_1,\lambda_2}$.

\begin{rem}
\label{equalfibers}
    In Proposition \ref{important} we made a particular choice of the cross-ratio. In particular, $$\Delta_{\lambda_1,\lambda_2} = \Delta_{\frac{1}{\lambda_1}, \frac{1}{\lambda_2}} = \Delta_{\frac{1}{1-\lambda_1}, \frac{1}{1-\lambda_2}} = \Delta_{1-\lambda_1, 1-\lambda_2} = \Delta_{\frac{\lambda_1}{\lambda_1-1}, \frac{\lambda_2}{\lambda_2-1}} =\Delta_{\frac{\lambda_1-1}{\lambda_1}, \frac{\lambda_2-1}{\lambda_2}}.$$
\end{rem}

To give a geometric description of the set $\Delta_{\lambda_1,\lambda_2}$, we need the following lemma.


\begin{lem}
\label{quadrics}
    Let $\lambda_1, \lambda_2 \in \mathbb{C}\setminus \{0,1\}$ be two distinct complex numbers and $Q_1, Q_2 \subset \mathbb{P}^3$ be quadrics defined by the equations $$q_1 := (t_2 - t_4)(t_3 - t_1) - \lambda_1(t_2 - t_1)(t_3 - t_4)$$ and $$q_2 := (t_2 + t_4)(t_3 + t_1) - \frac{\lambda_2}{\lambda_1}(t_2 + t_1)(t_3 + t_4)$$ respectively. Then $Q_1 \cap Q_2$ is a normal elliptic curve.
\end{lem}

\begin{proof}
Clearly $E := Q_1 \cap Q_2$ is a complete intersection, so $E$ is a curve.
    Let $\lambda_2' := \frac{\lambda_2}{\lambda_1} \neq 1$. We have to show that the corresponding Jacobian matrix $M$ has rank 2 at all points of $E$: $$M = \begin{pmatrix}
        (t_4 - t_2) + \lambda_1(t_3 - t_4) & (t_3 - t_1) + \lambda_1(t_4 - t_3) & (t_2 - t_4) + \lambda_1(t_1 - t_2) & (t_1 - t_3) + \lambda_1(t_2 - t_1) \\
        (t_2 + t_4) - \lambda_2'(t_3 + t_4) & (t_1 + t_3) - \lambda_2'(t_3 + t_4) & (t_2 + t_4) - \lambda_2'(t_1 + t_2) & (t_1 + t_3) - \lambda_2'(t_2 + t_1)
    \end{pmatrix}$$
    Assume that there is a point $p = [t_1:t_2:t_3:t_4] \in E$ such that $M$ has rank smaller than 2 at $p$, and denote the rows of $M$ by $v_1$ and $v_2$. We know that $v_1 = Cv_2$ for some $C$. Summing up the coordinates of $v_1$ we get zero, hence the same must hold for $v_2$. However, the sum for $v_2$ is equal to $$(2 - 2\lambda_2')(t_1 + t_2 + t_3 + t_4) = 0,$$ which implies that $$t_1 + t_2 + t_3 + t_4 = 0.$$ Now, the sum of the first and the fourth coordinates of $v_2$ is zero, hence $$(1 - \lambda_1)(t_4 + t_1 - t_2 - t_3) = 0.$$ Thus, $$t_2 + t_3 = t_4 + t_1 = 0.$$
    Comparing the sums of the first and third coordinates we get $$(C + \lambda_1)(t_2 + t_4) = 0.$$ However, if $t_2 + t_4 = 0$ we get $t_2 = t_1 = -t_3 = - t_4$, and we conclude that $t_1 = t_2 = t_3 = t_4 = 0$ from $S \in Q_2$. Therefore, $C = -\lambda_1$. Finally, comparing the sums of the first and second coordinates we get $$(1 + C\lambda_2')(t_3 + t_4) = 0.$$ We can show that $t_3 + t_4 = 0$ implies $t_1 = t_2 = t_3 = t_4 = 0$ analogously to the above, hence $C\lambda_2' = -1$ holds. Together with $C = -\lambda_1$ we get $\lambda_2 = 1$, which is impossible. Thus, the intersection $E$ is smooth. By adjunction formula the arithmetic genus of $E$ is equal to 1, hence it is an elliptic normal curve of degree $4$.
\end{proof}  


We are ready to give a description of fibers of the Prym map $\mathcal{P}_2[4]$.

\begin{thm}
\label{discr}
    Each fiber $\Delta_{\lambda_1,\lambda_2}$ with $\lambda_1, \lambda_2 \notin \{-1, \frac12, 2, e^{\pi i /3}, e^{-\pi i /3}\}$ of $\mathcal{P}_2[4]$ is isomorphic to the projective line $\mathbb{P}^1$ without 8 points.
\end{thm}

\begin{proof}

    Let $Q_1$ and $Q_2$ be quadrics in $\mathbb{P}^3$ defined in Lemma \ref{quadrics} and $E = Q_1 \cap Q_2$ be the elliptic normal curve in $\mathbb{P}^3$. It follows from the proof of Lemma \ref{difflambdas} that for any point $[\{t_1, t_2, t_3\}] \in \Delta_{\lambda_1, \lambda_2}$ we have $[t_1:t_2:t_3:1] \in E$. Therefore, the fiber $\Delta_{\lambda_1, \lambda_2}$ consists of equivalence classes of points in $E \cap \{t_4 = 1\}$ satisfying the constraints imposed on triples in $\Delta_2$. First, it follows from the assumption $\lambda_1, \lambda_2 \notin \{-1, \frac12, 2, e^{\pi i /3}, e^{-\pi i /3}\}$ that there is no non-trivial permutation $\sigma \in S_3$ such that both (ordered) triples $(t_1,t_2,t_3)$ and $(t_{\sigma(1)},t_{\sigma(2)}, t_{\sigma(3)})$ satisfy the system \ref{syst}. Moreover, the map $$\iota: [t_1:t_2:t_3:t_4] \mapsto \left[\frac{1}{t_1}:\frac{1}{t_2}: \frac{1}{t_3}: \frac{1}{t_4}\right]$$ induces a rational involution of $E = Q_1 \cap Q_2$ which we denote by $\iota_{|E}$. Since $E$ is a smooth algebraic curve, $\iota_{|E}$ can be extended to an involution of $E$ which we denote by the same letter. The fixed points of $\iota_{|E}$ are $[-1:1:1:1], [1:-1:1:1], [1:1:-1:1], [1:1:1:-1]$ and we find the natural quotient to be $\pi: E \to E/\langle\iota\rangle \simeq \mathbb{P}^1$. 
    
    Define the following subvarieties of $\mathbb{P}^3$: $$A_{ij} := \{[t_1:t_2:t_3:t_4] \mid t_i^2 = t_j^2 \text{ for } i\neq j \text{ and } i,j\in\{1,2,3\}\},$$ $$A_4 := \{[t_1:t_2:t_3:t_4] \mid t_4 = 0\}$$ and $$B_{i}^k := \{[t_1:t_2:t_3:t_4] \mid t_i = k \text{ for }(i,k) \in \{1,2,3\} \times \{0,1\}\}.$$ 
    Let $S$ be the union of these subvarieties. Then it follows by the above and the definition of $\Delta_2$ that $\Delta_{\lambda_1, \lambda_2} \simeq (E \setminus (E \cap S))/\langle\iota\rangle$. 
    
    We find that $E \cap S$ is the union of 12 points: 4 coordinate points, 4 fixed points of $\iota_E$ and 4 non-trivial points, one of each in the intersection of $E$ and the coordinate space $\{t_i = 0\}$. Since the eight points which are not fixed by $\iota_E$ in $E \cap S$ are precisely the indeterminacy locus of $\iota$, they must form complete orbits under the action of the extended involution $\iota_E$, hence $\Delta_{\lambda_1, \lambda_2}$ is isomorphic to $\mathbb{P}^1 \setminus \pi(E \cap S)$, where $|\pi(E \cap S)| = 8$.

\end{proof}

\begin{rem}
    There are two exceptional fibers of $\mathcal{P}_2[4]$ that Theorem \ref{discr} does not cover, namely $\Delta_{-1, \frac{1}{2}}$ and $\Delta_{e^{\pi i /3}, e^{-\pi i /3}}$ (note that all other combinations of $(\lambda_1, \lambda_2)$ give the same fibers by Remark \ref{equalfibers}). Both fibers contain sets $\{t_1,t_2,t_3\} \in \Delta_2$ such that there exists a non-trivial permutation $\sigma \in S_3$ with $[t_1:t_2:t_3:1], [t_{\sigma(1)}: t_{\sigma(2)}: t_{\sigma(3)}:1] \in Q_1 \cap Q_2$, hence geometrically these fibers are projective lines with some points removed and some points glued. 
\end{rem}

\section{The Prym map on $\mathcal{RH}_g[4]^{hyp}$}
\label{sec5}

In this section we will prove that the Prym map $$\mathcal{P}_g[4]: \mathcal{RH}_g[4]^{hyp} \to \mathcal{A}_{3g-3}^{(1,\ldots,1,4,\ldots4)}$$ is injective for $g \geq 3$.

We start with the following result, which we state in our setting.

\begin{prop}[\cite{borówka2025prymmapscycliccoverings}, Proposition 2.1]
\label{group}
    Let $(P, \Theta_P)$ be an element of $\Image\mathcal{P}_g[4]$ and $[f: C \to H] \in \mathcal{RH}_g[4]^{hyp}$ be any preimage of $(P, \Theta_P)$. Let $\sigma$ be the deck transformation of $C$ inducing $f$ and $j$ be a lift of the hyperelliptic involution on $H$ to $C$.
    Then the subgroup of automorphisms $$G := \{ \phi \in \Aut(P,\Theta_P) \: | \: \phi(x) = x \:\: \forall x \in K(\Theta_P)\}.$$
    is isomorphic to $\langle \sigma_, -j \rangle \simeq D_4$, where $\sigma$ and $j$ are induced by the corresponding automorphisms of $C$.
\end{prop}

\begin{thm}
\label{injective}
    Let $g \geq 3$ be a positive integer. Then the Prym map $\mathcal{P}_g[4]: \mathcal{RH}_g[4]^{hyp} \to\mathcal{A}_{3g-3}^{(1,\ldots,1,4,\ldots4)} $ is injective.
\end{thm}

\begin{proof}
    Let $(P, \Theta_P)$ be an element of $\Image\mathcal{P}_g[4]$. By Proposition \ref{group} we can recover $\langle \sigma_, -j \rangle \simeq D_4$. Take any cover $[f: C \to H]$ in the preimage of $(P, \Theta_P)$ and let $\sigma, j$ be the corresponding automorphisms of $C$. Let $\pi_j: C \to C_j$ and $\pi_{j\sigma}:C \to C_{j\s}$ be the quotients. Then $\{g(C_{j}), g(C_{j\s})\} = \{2g-2, g-1\}$, hence we can take $j$ to be the involution on $C$ such that $\dim \pi_j^*JC_j = g-1$, and we have $\pi_j^*JC_j = \Image(1+j) \subset P$. Note that $\pi_j^*$ is an embedding by \cite[Proposition 11.4.3]{BL}, hence we can recover $JC_j$ inside the Prym variety $P$. It follows from the Torelli Theorem that we can recover the curve $C_j$ as well. Let $\pi_{j\s, \s^2}:C_{j\s} \to C_{j\s, \s^2}$ so that $\Image((1+j\s)(1+\s^2)) = (\pi_{j\s,\s^2}\circ\pi_{j\s})^*JC_{j\s, \s^2} \subset P$. Since $\pi_{j\s,\s^2}\circ\pi_{j\s}: C \to C_{j\s, \s^2}$ is the Galois quotient by $\langle j\s, \s^2 \rangle$ and both $C_{\sigma^2} \to C_{j\s, \s^2}$,  $C_{j\s} \to C_{j\s, \s^2}$ are ramified, the pullback $(\pi_{j\s,\s^2}\circ\pi_{j\s})^*$ is also an embedding. Therefore, we can recover the curve $C_{j\s, \s^2}$ in an analogous way. 

    It follows from Lemma \ref{equations} that there exists a tuple $T = \{t_1, \ldots, t_{2g-1}\} \in \Delta_g$ such that $C_{j}$ and $C_{j\s, \s^2}$ are hyperelliptic curves whose sets of Weierstrass points are $\{[1:1], [t_1:1], \ldots, [t_{2g-1}:1]\}$ and $\{[1:1], [t_1^2:1], \ldots, [t_{2g-1}^2:1]\}$ respectively. It follows from Theorem \ref{moduli} and Lemma \ref{geomprym} part i) that $\mathcal{P}_g[4]$ is generically injective if for any tuple $S = \{s_1, \ldots, s_{2g-1}\} \in \Delta_g$, the existence of Möbius transformations $u, v: \mathbb{P}^1 \to \mathbb{P}^1$ such that $$u([1:1]) = [1:1], v([1:1]) = [1:1], u([t_i:1]) = [s_i:1], v[t_i^2:1] = [s_i^2:1]$$ for all $i \in \{1,\ldots, 2g-1\}$ implies 
    $S \sim T$ in $\Delta_g$. 

    Assume that we are given such maps $u$ and $v$ for some $S \in \Delta_g$. Then for any $t_i \in T$ we have $u(t_i)^2 = v(t_i^2)$. Since both $u(x)^2$ and $v(x^2)$ are rational functions in $x$ with numerator and denominator of degree 2, the function $u(x)^2 - v(x^2)$ is either identically zero or has at most 4 roots. However, the fact that $g \geq 3$ implies that the latter is impossible so we have $u(x)^2 = v(x^2)$ for all $x$. In particular, $u(x) = u(-x)$ for all $x$, and the only Möbius transformations with this property mapping $[1:1]$ to itself are $u(x) = x$ and $u(x) = \frac{1}{x}$. This implies that $S \sim T$ in $\Delta_g$, which finishes the proof.
      
\end{proof}

\bibliographystyle{alpha}
\bibliography{sample}

\textsc{A. Shatsila, Doctoral School of Exact and Natural Sciences, Jagiellonian University, ul. prof. Stanisława Łojasiewicza 6, 30-348 Kraków, Poland}\\
\textit{email address:} anatoli.shatsila@doctoral.uj.edu.pl.

\end{document}